\documentclass[a4paper]{amsart}
\usepackage{amssymb,amsxtra}

\theoremstyle{plain}
\newtheorem{theorem}{Theorem}[section]
\newtheorem{lemma}[theorem]{Lemma}
\newtheorem{corollary}[theorem]{Corollary}
\newtheorem{proposition}[theorem]{Proposition}
\theoremstyle{definition}
\newtheorem{definition}[theorem]{Definition}
\theoremstyle{remark}
\newtheorem{remark}[theorem]{Remark}

\DeclareMathOperator{\home}{home}
\DeclareMathOperator{\trop}{Trop}
\DeclareMathOperator{\ptrop}{Trop_{>0}}
\DeclareMathOperator{\nntrop}{Trop_{\geq 0}}
\DeclareMathOperator{\init}{in}
\DeclareMathOperator{\sat}{Sat}
\DeclareMathOperator{\Hom}{Hom}
\newcommand{\Q}{\mathbb{Q}}
\newcommand{\Qi}{\overline{\mathbb{Q}}}
\newcommand{\R}{\mathbb{R}}
\newcommand{\Ri}{\overline{\mathbb{R}}}
\newcommand{\C}{\mathbb{C}}
\newcommand{\Z}{\mathbb{Z}}
\newcommand{\V}{\mathcal{V}}

\begin{document}
\title[Universal valued fields]{Universal valued fields and lifting points
in local tropical varieties}
\author{D.~A. Stepanov}
\address{The Department of Mathematical Modelling \\
         Bauman Moscow State Technical University \\
         2-ya Baumanskaya ul. 5, Moscow 105005, Russia}
\email{dstepanov@bmstu.ru}
\thanks{The research was supported by the Russian Grant for Leading 
Scientific Schools no. 5139.2012.1, and RFBR grant no. 11-01-00336-a.}
\date{April 29, 2013}

\begin{abstract}
Let $k$ be a field with a real valuation $\nu$ and $R$ a $k$-algebra. We show
that there exist a $k$-algebra $K$ and a real valuation $\mu$ on $K$ extending 
$\nu$ such that any real ring valuation of $R$ is induced by $\mu$ via some 
homomorphism from $R$ to $K$; $K$ can be chosen to be a field. Then we study the 
case when $\nu$ is trivial and $R$ a complete local Noetherian ring with the 
residue field $k$. Let $K$ be the ring $\bar{k}[[t^\R]]$ of Hahn series with 
its natural valuation $\mu$; $\bar{k}$ is an algebraic closure of $k$. 
Despite $K$ is not universal in the strong sense defined above, 
it has the following weak universality property: for any local valuation $v$ and 
a finite set of elements $x_1,\ldots,x_n$ of $R$ there exists a homomorphism 
$f\colon R\to K$ such that $v(x_i)=\mu(f(x_i))$, $i=1,\ldots,n$. If $R=
k[[x_1,\ldots,x_n]]/I$ for an ideal $I$, this property implies that every point 
of the local tropical variety of $I$ lifts to a $K$-point of $R$. Similarly,
if $R=k[x_1,\ldots,x_n]/I$ is a finitely generated algebra over $k$, lifting
points in the tropical variety of $I$ can be interpreted as the weak universality
property of the field $\bar{k}((t^\R))$ of Hahn series.
\end{abstract}

\maketitle

\section{Introduction}

Let $k$ be a field and $R$ a commutative $k$-algebra with unity. A \emph{real} 
(or \emph{rank one}) \emph{valuation} on $R$ is a function $v\colon R\to \R\cup 
\{+\infty\}$ such that $v(1)=0$, $v(0)=+\infty$, $v(xy)=v(x)+v(y)$ and 
$v(x+y)\geq \min\{v(x),v(y)\}$ for all $x$, $y\in R$. In this paper we consider 
only real ring valuations. Assume that $k$ is endowed with some valuation $\nu$ 
and $K$ is some $k$-algebra with a valuation $\mu$ extending $\nu$. Then every
homomorphism $f\colon R\to K$ of $k$-algebras induces a valuation $\mu\circ f$
of $R$ extending $\nu$. The first question we address in this paper is the 
following: given $R$ and $\nu$, does there exist a universal valued field $K$
such that \emph{any} valuation of $R$ that extends $\nu$ is obtained via some
homomorphism $f\colon R\to K$ as described above? We give an affirmative answer 
to this question in Theorem~\ref{T:ufield} of Section~\ref{S:ufield}. Note that 
the universal field we are looking for is different from the maximal immediate 
extension of $k$ (\cite[p. 191]{Krull}) since we do not insist that $v$ must have 
the same residue field as $\nu$.

The second problem that we study in this work is lifting points in local
tropical varieties. If $R$ is a local ring with the maximal ideal $\mathfrak{m}$,
we call a valuation $v$ of $R$ \emph{local}, if $v$ is nonnegative on $R$
and positive on $\mathfrak{m}$. Note that if $R$ contains a field $k$, any
local valuation of $R$ must be trivial on $k$, i.e., $v(x)=0$ for all $x\in k$,
$x\ne 0$. Now let $R=k[[x_1,\ldots,x_n]]$ be the ring of formal power series
in variables $x_1,\ldots,x_n$, and $I\subset R$ an ideal. The \emph{local 
tropicalization} $\ptrop(I)$ of $I$ is the set
$$\ptrop(I)=\{(v(x_1),\ldots,v(x_n)\,|\,v \text{ is local},\; v|_I=+\infty\}
\subseteq \Ri^n,$$
where $v$ runs over all local valuations of $R$ that take value $+\infty$ on
the ideal $I$, and $\Ri=\R\cup\{+\infty\}$. This is a particular case of
\cite[Definition~6.6]{PPS}; see also Section~\ref{S:liftpoint}. Let us introduce 
two special fields. Denote by $\bar{k}$ some algebraic closure of $k$. The 
field $K=\bar{k}((t^Q))=\bigcup_{N\geq 1}\bar{k}((t^{1/N}))$ of \emph{Puiseux 
series} with coefficients in $\bar{k}$ is the set of all formal sums
$$\sum_{m\in\Q} a_m t^m,$$ 
where $a_m\in \bar{k}$, the set $\{m\,|\,a_m\ne 0\}\subset\Q$ is bounded from
below, and all its elements have bounded denominators. The field $K'=
\bar{k}((t^\R))$ of \emph{Hahn series} is the set of all formal 
sums
$$a(t)=\sum_{m\in \R} a_m t^m,$$
where the set $\{m\,|\,a_m\ne 0\}$ is well-ordered, that is, any its subset has 
the least element. % This field was introduced by T. Markwig in \cite{Mar}. 
Addition and multiplication in $K$ and $K'$ are defined in a natural way. 
These fields carry also a natural valuation $\mu$ determined by $\mu(t)=1$. 
The \emph{ring of Puiseux} (respectively \emph{Hahn}) \emph{series} 
$\mathcal{O}=\bar{k}[[t^\Q]]$ (resp., $\mathcal{O}'=\bar{k}[[t^\R]]$) is the 
valuation ring of $\mu$, i.e., the subring of $K$ (resp. $K'$) where $\mu$ is 
nonnegative. In analogy with the theory of usual non-local tropicalization 
(see, e.g., \cite{JMM}), we say that a point $w=(w_1,\ldots,w_n)\in 
\ptrop(I)\cap\Qi^n$, where $\Qi=\Q\cup\{+\infty\}$, \emph{lifts to} $\mathcal{O}$, 
if there exists a homomorphism $f\colon R\to K$ such that $w_i=\mu(f(x_i))$ for 
all $i=1,\ldots,n$. Lifting of a point $w\in\ptrop(I)$ to $\mathcal{O}'$ can 
be defined similarly. In the non-local theory of tropicalization, it is known 
that if characteristic of $k$ is $0$, then each rational point of the 
tropicalization of a variety admits a lifting to a $K$-point of the variety. 
Several proofs and generalizations are given in \cite{Dra}, \cite{JMM}, 
\cite{Katz}, \cite{P}. For $k=\C$ and local tropical varieties, a similar 
result was announced by N. Touda in \cite{To}, however, a complete proof 
did not appear. 

The possibility to lift points of local tropical varieties to $\mathcal{O}$
or $\mathcal{O}'$ means that these rings play a role of a kind of universal 
domains with respect to valuations on local $k$-algebras $R$, where the field 
$k$ is trivially valued. This is not the strong universality introduced
in the first paragraph of this Introduction. Indeed, if we set $R=k[[x,y]]$
and $v$ to be a monomial valuation defined by $v(x)=v(y)=1$, then it is easy
to see that $v$ is not induced by any homomorphism from $R$ to $\mathcal{O}$ or 
$\mathcal{O}'$. However, it is easy to define a homomorphism $f\colon R\to 
\mathcal{O}$, say, by sending $x$ and $y$ to $t$, such that $\mu(f(x))=v(x)$,
$\mu(f(y))=v(y)$. Precisely, we have the following property, which we call
\emph{weak universality} of the rings of Puiseux and Hahn series. 
Assume that $R$ is a complete local Noetherian ring, $k$ is a field isomorphic to 
the residue field of $R$ and contained in $R$, and $x_1,\ldots,x_n\in R$ a finite 
collection of elements. Then for each local valuation $v$ of $R$ there exists a 
homomorphism $f\colon R\to \mathcal{O}'$ such that $v(x_i)=\mu(f(x_i))$ for 
all $i=1,\ldots,n$. If, moreover, $k$ has characteristic $0$ and $v$ takes
only rational values, then there exists a homomorphism $f\colon R\to \mathcal{O}$ 
with the same property. Similarly, lifting points in the non-local tropical 
varieties can be expressed as the weak universality of the \emph{fields} of 
Puiseux and Hahn series with respect to valuations on finitely generated 
$k$-algebras $R$. We prove these properties in Section~\ref{S:wufield} 
and then apply them to the local tropicalization in Section~\ref{S:liftpoint}.

Our proof of the weak universality of the Puiseux and Hahn series rings is closed 
to the proof of the lifting points property of tropical varieties given in 
\cite{JMM}. We also use a descent by dimension, but we descent not to dimension 
$0$ but to dimension $r$ equal to the rational rank of the value group of $v$; 
this allows to work always over the field $\bar{k}$ and not to pass to varieties 
over the field of Puiseux series.

I thank B.~Teissier and M.~Matusinski for their comments on the draft of this
work.

\section{The universal field}\label{S:ufield}

In this section we show existence of the universal valued field for
a given valued ring $k$ and a $k$-algebra $R$. The proof is based on the
following theorem about extension of valuations to a tensor product. If $A$ is
a ring, $B$ is an $A$-algebra, and $v$ is a valuation of $B$, then $v$ induces
a valuation of $A$ which we call the \emph{restriction} of $v$ and denote
$v|_A$.

\begin{theorem}\label{T:tensor}
Let $A$ be a commutative ring with unity and $B$ and $C$ $A$-algebras, also with
unity. Let $u$ be a valuation of $B$, $v$ a valuation of $C$, and assume that
$u|_A=v|_A$. Then there exists a valuation $w$ of $B\bigotimes_A C$ that extends
both $u$ and $v$: $w|_B=u$, $w|_C=v$, where $B$ and $C$ map to $B\bigotimes_A C$
as $B\ni b\mapsto b\otimes 1$, $C\ni c\mapsto 1\otimes c$.
\end{theorem}

This is stated without a proof in \cite[1.1.14f]{H}. Below we give two proofs of
Theorem~\ref{T:tensor}. The \emph{first proof} follows from recent stronger
results of I. B. Yaacov \cite{Itai} and J. Poineau \cite{Poi}. It is known that
in the conditions of Theorem~\ref{T:tensor}, the tensor product $B\bigotimes_A C$
carries a natural seminorm $\|\cdot\|$:
$$\|z\|=\inf \max_i (\exp(-u(x_i)-v(y_i))),\quad z=\sum_i x_i\otimes y_i,$$
where the infinum is taken over all representations of $z\in B\bigotimes_A C$ as
$\sum x_i\otimes y_i$, $x_i\in B$, $y_i\in C$.

\begin{theorem}[{\cite[Theorem 6]{Itai}}]\label{T:Itai}
Let $k$ be a valued algebraically closed field, and $K$ and $L$ two field 
extensions of $k$. Assume that $K$ and $L$ are endowed with valuations $u$ and
$v$ that restrict to the given valuation of $k$. Then, the natural seminorm 
$\|\cdot\|$ of $K\bigotimes_k L$ is multiplicative and the function 
$-\log\|\cdot\|\colon K\bigotimes_k L\to \Ri$ is a valuation extending both 
$u$ and $v$.
\end{theorem}

In \cite{Itai}, this theorem is proven with a help of non-standard technique 
(ultrapowers) and results on quantifier elimination in some formal theories. It
can also be deduced from \cite[Section 3]{Poi}, where the technique is the 
theory of affinoid algebras. To reduce Theorem~\ref{T:tensor} to 
Theorem~\ref{T:Itai}, assume first that the valuation $u$ or $v$ has a nontrivial
home. Let $\mathfrak{p}=\home(u)=\{x\in B|u(x)=+\infty\}$, $\mathfrak{q}=
\home(v)$. Denote by $\mathfrak{p}\otimes 1$ and $1\otimes\mathfrak{q}$ the
extensions of $\mathfrak{p}$ and $\mathfrak{q}$ to $B\bigotimes_A C$. Then
$u$ and $v$ can be considered as valuations on the rings $B/\mathfrak{p}$ and
$C/\mathfrak{q}$ respectively, and 
$$B/\mathfrak{p} \bigotimes_A C/\mathfrak{q}\simeq (B\bigotimes_A C)/
(\mathfrak{p}\otimes 1 + 1\otimes\mathfrak{q})$$
and $B/\mathfrak{p} \bigotimes_{A/\mathfrak{p}\cap A} C/\mathfrak{q}$ are
isomorphic as $A$-algebras. It follows that we can assume from the beginning
that $A$, $B$, and $C$ are domains, and $\home(u)=\home(v)=\{0\}$. Thus,
$u$ and $v$ extend canonically to the fields of fractions of $B$ and $C$.
Then, passing to the localizations (see, e.g., 
\cite[Propositions 3.5 and 3.7]{AM}), we can assume that $A=k$, $B=K$, and $C=L$
are fields, so it remains only to reduce to an algebraically closed field $k$.
Let $\bar{k}$ be an algebraic closure of $k$. Since $K'=K\bigotimes_k \bar{k}$
and $L'=L\bigotimes_k \bar{k}$ are integral extensions of $K$ and $L$ 
respectively, there exist a prime ideal $\mathfrak{p}\subset K'$ lying over 
$\{0\}$ and a prime ideal $\mathfrak{q}\subset L'$ lying over $\{0\}$. Now,
let $\Bar{K}$ and $\Bar{L}$ be the fields of fractions of 
$K'/\mathfrak{p}$ and $L'/\mathfrak{q}$ respectively. $\Bar{K}$ is an
algebraic extension of $K$, thus, by \cite[Chapter XII, \S 3]{Lang}, $u$ extends 
to $\Bar{K}$, and similarly $v$ extends to $\Bar{L}$. The tensor product 
$K\otimes_k L$ maps to $\Bar{K}\otimes_{\bar{k}}\Bar{L}$, and this map induces 
embeddings of valued fields $K\subseteq \Bar{K}$ and $L\subseteq \Bar{L}$. Thus,
Theorem~\ref{T:tensor} indeed follows from Theorem~\ref{T:Itai}.

The \emph{second proof} of Theorem~\ref{T:tensor} is based on the results of 
G. M. Bergman. This proof uses a more standard argument of commutative algebra, so 
we think it is of independent interest. A \emph{pseudovaluation} on a ring $R$ is 
a map $v\colon R\to\R\cup\{+\infty\}$ satisfying the same axioms as a valuation 
with the exception that instead of $v(xy)=v(x)+v(y)$ we require only 
$v(xy)\geq v(x)+v(y)$. 

\begin{theorem}\label{T:Bergman}
Let $p$ be a pseudovaluation on a commutative ring $R$ with unity, and $S$ a 
multiplicative subsemigroup in $(R,\cdot)$ such that $p|_S$ is a semigroup 
homomorphism from $S$ to $\Ri$. Let $I$ be an ideal of $R$ such that there 
is no $s\in S$, $f\in I$ satisfying $v(s)=v(f)< v(f-s)$. Then there exists a 
valuation $v\geq p$ on $R$ such that $v|_I=+\infty$, $v|_S=p|_S$.
\end{theorem}
\begin{proof}
This theorem is a direct generalization of \cite[Corollary 1]{Berg}. Indeed, the
formula $q(x)=\sup_{f\in I} p(x+f)$ defines a pseudovaluation on $R$. On the
semigroup $S$, this pseudovaluation coincides with $p$. Then, by 
\cite[Theorem 2]{Berg} there exists a valuation $v\geq q$ on $R$ that coincides 
with $q$ on $S$. Since $q|_I=+\infty$, we have also $v|_I=+\infty$.
\end{proof}

We continue with an auxiliary lemma.

\begin{lemma}\label{L:fgtensor}
Let $B$ and $C$ be two finitely generated $A$-algebras over a ring $A$. Let $u$
and $v$ be valuations of $B$ and $C$ respectively that induce the same valuation
on $A$. Finally, let $x_1,\ldots,x_m\in B$ and $y_1,\ldots,y_n\in C$ be two fixed 
collections of elements. Then there exists a valuation $w$ on $B\bigotimes_A C$
such that for all $i=1,\ldots,m$ $w(x_i\otimes 1)=u(x_i)$, and for all
$j=1,\ldots,n$ $w(1\otimes y_j)=v(y_j)$.
\end{lemma}
\begin{proof}
\emph{Step 1.} First, by the same argument as above we can reduce to the case 
when $A=k$ is a field, $B$ and $C$ are finitely generated domains over $k$, and 
$\home(u)=\{0\}$, $\home(v)=\{0\}$. 
Next, adding or eliminating some elements if necessary, we can suppose that 
$x_i$ generate $B$ and $y_j$ generate $C$ as rings over $k$, and none of $x_i$, 
$y_j$ is $0$. Represent $B$ as a quotient of a polynomial ring 
$k[X_1,\ldots,X_m]/I$ and $C$ as $k[Y_1,\ldots,Y_n]/J$ for ideals
$I\subset k[X]=k[X_1,\ldots,X_m]$ and $J\subset k[Y]=k[Y_1,\ldots,Y_n]$; here 
$x_i=X_i\mod I$, $y_j=Y_j\mod J$. 

\emph{Step 2.} Let us recall some terminology. We denote also by $u$ the 
valuation $u$ of $B$ restricted to $k$ ($=$ $v$ of $C$ restricted to $k$). If 
$w=(w_1,\ldots,w_m)\in\R^n$ is a vector and $f=\sum a_M X^M\in k[X]$ is a 
polynomial in $m$ variables, the number 
$$w(f)=\min_M (u(a_M)+\langle w,M\rangle),$$
where $\langle w,M\rangle=\sum w_iM_i$, is called \emph{$w$-order} of $f$. The
polynomial 
$$\init_w(f)= \sum_{u(a_M)+\langle w,M\rangle=w(f)} a_M X^M$$
is called the \emph{$w$-initial form} of $f$. For an ideal $I\subset k[X]$,
the ideal $\init_w I=(\init_w(f)|f\in I)$ generated by all $w$-initial forms
of elements of $I$ is called the \emph{$w$-initial ideal} of $I$.

Now let $w^1=(u(x_1),\ldots,u(x_m))\in\R^m$, $w^2=(v(y_1),\ldots,w(y_n))\in\R^n$.
Note that since $u$ and $v$ take finite values on $x_i$ and $y_j$, the initial
ideals $\init_{w^1} I$ and $\init_{w^2} J$ are monomial free. Denote $k[X,Y]=
k[X_1,\ldots,X_m,Y_1,\ldots,Y_n]$, $w=(w^1,w^2)\in\R^{m+n}$, and let $I\otimes 1$ 
and $1\otimes J$ be the extensions of the ideals $I$ and $J$ to $k[X,Y]$ 
respectively. Then, $\init_w(I\otimes 1)=\init_{w^1}(I)\otimes 1$, and similarly
for $J$. Indeed, consider $f=\sum f_ib_i\in I\otimes 1$, where $f_i\in k[X,Y]$,
$b_i\in I$. Write $f$ as a polynomial in $Y$: $f=\sum_N b_N Y^N$, where
$b_N=b_N(X)\in I$, $b_N\ne 0$. Consider the combination
\begin{equation}\label{E:comb1}
\sum_{w^1(b_N)+\langle w^2,N\rangle\text{ is minimal}} \init_{w^1}(b_N) Y^N.
\end{equation}
If it is $0$, then this holds identically in $Y$, and thus $\init_{w^1}(b_N)=0$
for all $b_N$ involved in \eqref{E:comb1}. But this contradicts the assumption
$b_N\ne 0$. It follows that the $w$-initial form of $f$ is a $k[X,Y]$-linear
combination of initial forms of $b_N\in I$.

\emph{Step 3.} Our next claim is that the initial ideal $\init_w(I\otimes 1+
1\otimes J)$ coincides with $\init_{w^1}(I)\otimes 1+1\otimes\init_{w^2}(J)$.
The inclusion $\supseteq$ is clear, so let us prove the inclusion $\subseteq$.
Consider an element
$$h=\sum_i f_ib_i + \sum_j g_jc_j \in I\otimes 1+1\otimes J,$$
where $f_i,g_j\in k[X,Y]$, $b_i\in I$, $c_j\in J$. We can assume that $b_i$
generate the ideal $I$, $c_j$ generate the ideal $J$, the initial forms
$\init_{w^1}(b_i)$ generate the initial ideal $\init_{w^1}I$, and 
$\init_{w^2}(c_j)$ generate $\init_{w^2}J$. If the initial forms do not
cancel, that is the combination
$$\sum \init_w(f_i)\init_{w^1}(b_i) + \sum \init_w(g_j)\init_{w^2}(c_j)$$
is not $0$, then it is the $w$-initial form $h_w$ of $h$, and there is nothing
to prove. If this combination is $0$, it follows that 
$$\sum \init_w(f_i)\init_{w^1}(b_i),\: \sum \init_w(g_j)\init_{w^2}(c_j)\in
\init_w(I\otimes 1)\cap\init_w(1\otimes J),$$
the last intersection being equal to the product $\init_w(I\otimes 1)\cdot
\init_w(1\otimes J)$ by Corollary~\ref{C:product} below. Thus, using Step 2,
$\sum \init_w(f_i)\init_{w^1}(b_i)=-\sum \init_w(g_j)\init_{w^2}(c_j)$ can
be written as $\sum h_{ij}\init_{w_1}(b_i)\init_{w^2}(c_j)$, $h_{ij}\in k[X,Y]$.
Then, consider
$$h=\sum f_ib_i + \sum g_jc_j \pm \sum h_{ij}b_ic_j =
\sum_i (f_i-\sum_j h_{ij}c_j)b_i + \sum_j (g_j+\sum_i h_{ij}b_i)c_j.$$
The $w$-order of $h$ is fixed, while $w$-orders of expressions in
parenthesis on the right have increased comparing with the $w$-orders of
$f_i$ and $g_j$. In this way we eventually rewrite $h$ in the form where
$w$-initial forms do not cancel, thus, $w$-initial form of $h$ is a combination
of $w$-initial forms of elements of $I$ and $J$.

\emph{Step 4.} Now we show that the ideal $\init_w(I\otimes 1+1\otimes J)$ 
is monomial free. Suppose that a monomial $X^MY^N$ can be represented as
\begin{equation}\label{E:comb2}
X^MY^N=\sum f_i\init_{w^1}(b_i)+\sum g_j\init_{w^2}(c_j),
\end{equation}
where $f_i$, $g_j\in k[X,Y]$, $b_i\in I$, $c_j\in J$. Since $\init_{w^1}I$
and $\init_{w^2}J$ are monomial free, they have points $x^0\in(\bar{k}^*)^m$,
$y^0\in(\bar{k}^*)^n$ respectively. Substituting the point $(x^0,y^0)\in
(\bar{k}^*)^{m+n}$ to \eqref{E:comb2}, we get a contradiction.

\emph{Step 5.} Finally, we construct a valuation $w$ on $B\bigotimes_k C$. First, 
consider a monomial valuation $w'$ on $k[X,Y]$ defined by $w'(X_i)=u(x_i)$,
$i=1,\ldots,m$, $w'(Y_j)=v(y_j)$, $j=1,\ldots,n$. Let $S$ be a semigroup
generated by all the monomials $aX^MY^N$, $a\in k$, $M\in\Z_{\geq 0}^{m}$,
$N\in\Z_{\geq 0}^{n}$. By Step 4 and Bergman's Theorem~\ref{T:Bergman}, $w'$ 
can be pushed forward to a valuation $w$ on
$$k[X,Y]/(I\otimes 1+1\otimes J)\simeq k[X]/I \bigotimes_k k[Y]/J \simeq
B\bigotimes_k C,$$
and $w$ has all the required properties.
\end{proof}

The following two results were needed for the proof of Lemma~\ref{L:fgtensor}.
We use the notation introduced in Step 2 of the proof of Lemma~\ref{L:fgtensor}
and assume some familiarity of the reader with Gr{\"o}bner bases.

\begin{lemma}\label{L:Groebner}
Let $I\subset k[X]$ be an ideal, and $\{b_1,\ldots,b_s\}$ a Gr{\"o}bner basis of
$I$ with respect to some monomial ordering $\leq$ on $k[X]$. Then $\{b_1,
\ldots,b_s\}$ is a Gr{\"o}bner basis for $I\otimes 1\subset k[X,Y]$ with respect
to any monomial ordering on $k[X,Y]$ that restricts to the ordering $\leq$
on $k[X]$.
\end{lemma}
\begin{proof}
For any $f=\sum f_i(X,Y)b_i(X)$ write $f=\sum_N(\sum a_j(X)b_j(X))Y^N$ as
a polynomial in $Y$. The leading monomial of $f$ is present only in one of the
expressions $(\sum a_jb_j)Y^N$, thus it is divisible by the leading monomial
of one of the $b_j$.
\end{proof}

\begin{corollary}\label{C:product}
Let $I\subset k[X]$ and $J\subset k[Y]$ be ideals. Then
$$(I\otimes 1)\cap(1\otimes J)=(I\otimes 1)\cdot(1\otimes J)$$
in $k[X,Y]$.
\end{corollary}
\begin{proof}
Fix a monomial ordering on $k[X,Y]$ and Gr{\"o}bner bases $\{b_i\}$ for $I$
and $\{c_j\}$ for $J$. Take $f\in(I\otimes 1)\cap(1\otimes J)$. By 
Lemma~\ref{L:Groebner}, we conclude that the leading monomial $aX^MY^N$, $a\in k$,
of $f$ is divisible by the leading monomial of one of $b_i$ and of one of $c_j$. 
But the leading monomial of each $b_i$ is coprime to the leading monomial of each 
$c_j$, thus $aX^MY^N$ is divisible by the leading monomial of some of the 
products $b_i(X)c_j(Y)$. This implies the corollary.
\end{proof}

Let us continue the proof of Theorem~\ref{T:tensor}. Now we consider the 
general case of the tensor product of $A$-algebras $B$ and $C$, not necessarily
finitely generated over $A$. It is again possible to reduce to the case when 
$A=k$ is a field, and this will be assumed in the sequel. Let us recall the 
construction of the tensor product. Denote by $B^\bullet$ and $C^\bullet$ the
multiplicative semigroups of the rings $B$ and $C$ respectively, and consider
the direct product of semigroups $S'=B^\bullet\times C^\bullet$. Then the
tensor product $B\bigotimes_k C$ can be identified with the quotient $k[S']/T$
of the semigroup algebra $k[S']$ by the ideal $T$ generated by all relations of
the form
\begin{align}\label{E:basis}
&(ax,y)-a(x,y),\: (x,ay)-a(x,y), \\ \notag
(x'+x'',y)- &(x',y)-(x'',y),\: (x,y'+y'')-(x,y')-(x,y''),
\end{align}
where $a\in k$, $x,x',x''\in B$, $y,y',y''\in C$. Note that these expressions 
generate $J$ not only as an ideal of $k[S']$ but also as a vector space 
over $k$.

The valuations $u$ and $v$ of $B$ and $C$ are semigroup homomorphisms
$u\colon B^\bullet\to \Ri$, $v\colon C^\bullet\to \Ri$. The rule 
$p(x,y)=u(x)+v(x)$ defines a semigroup homomorphism 
$p\colon S'\to \Ri$. Furthermore, we extend $p$ to a pseudovaluation on the 
semigroup algebra $k[S']$. For $f=\sum a_s s \in k[S']$, where $a_s\in k$, 
$s\in S'$, we set
$$p(f)=\min_{s\in S'} (u(a_s)+p(s))$$
(we could write $v(a_s)$ instead of $u(a_s)$). Let $S$ be a subsemigroup of
$k[S']$ generated by all the monomials $a(x,y)$, $a\in k$, $a\ne 0$, 
$(x,y)\in S$. It is clear that $p|_S$ is a semigroup homomorphism. 
Next we are going to show that there exists a valuation $\bar{w}$ on $k[S']$ 
that coincides with $p$ on $S$ and can be pushed forward to a valuation $w$ of 
the quotient $k[S']/T$. It suffices only to check the conditions of 
Theorem~\ref{T:Bergman}, i.e., if $f\in T$ and $s\in S$ is a monomial, then 
$p(f-s)\leq p(f)$. Represent $f$ as a $k$-linear combination $\sum a_i r_i$ of 
expressions $r_i$ of the form \eqref{E:basis}, and $s$ as $a(x_0,y_0)$, 
where $a\in k$, $x_0\in B$, $y_0\in C$. Let $x_1,\ldots,x_m$ be all the 
elements of $B$ and $y_1,\ldots,y_n$ all the elements of $C$ that are present 
in the monomials of $r_i$. Consider the finitely generated rings 
$B'=k[x_0,x_1,\ldots,x_m]$ and $C'=k[y_0,y_1,\ldots,y_n]$ and restrictions 
$u'$ and $v'$ to $B'$ and $C'$ respectively. If we represent $B'\bigotimes_k C'$ 
as a quotient of a semigroup algebra $k[(B')^\bullet\times(C')^\bullet]$, 
then $f$ can be considered as an element of the corresponding ideal $T'\subset k[(B')^\bullet\times(C')^\bullet]$. The pseudovaluation $p$ also naturally 
restricts to this semigroup algebra. On the other hand, we know by 
Lemma~\ref{L:fgtensor} that $u'$ and $v'$ extend simultaneously to 
$B'\bigotimes_k C'$, thus, $p$ achieves its minimum on the monomials of $f$ 
at least twice. This shows that Theorem~\ref{T:Bergman} applies to $p$, 
$S$ and $T$. We get a valuation $w$ on
$$k[S']/T\simeq B\bigotimes_k C$$
which is a simultaneous extension of $u$ and $v$. This finishes the second proof 
of Theorem~\ref{T:tensor}.

We are ready to describe the construction of universal valued fields. Now,
let $k$ be a ring with a real valuation $\nu$ and $R$ an arbitrary 
commutative $k$-algebra with unity. Denote by $\V(R,k)$ the set of all 
real ring valuations of $R$ that extend $\nu$.

\begin{theorem}\label{T:ufield}
Given $k$, $R$, and $\nu$, there exist a $k$-algebra $K$ and a valuation $\mu$ on 
$K$ such that $\mu|_k=\nu$, and for any $v\in\V(R,k)$ there exists a morphism 
$f_v\colon R\to K$ of $k$-algebras such that the valuation $v$ on $R$ is induced 
by the valuation $\mu$ on $K$ via the morphism $f_v$. 
\end{theorem}
\begin{proof}
Consider a $k$-algebra
$$K=\bigotimes_{v\in\V(R,k)} R,$$
the restricted tensor product over $k$ of the $k$-algebra $R$ with itself, one
copy for each $v\in\V(R,k)$, see \cite[p. 713, Proposition~A6.7b]{Eis}. Let us
show that there exists a valuation $\mu$ on $K$ that is a simultaneous
extension of all the valuations $v$ of $R$. This is a consequence of
Theorem~\ref{T:tensor} and Zorn's lemma. Indeed, for each subset $V\subseteq
\V(R,k)$ we have a natural homomorphism of $k$-algebras
$$K_V=\bigotimes_{v\in V} R \to K,$$
and, for $U\subseteq V$, a natural homomorphism $K_U\to K_V$.
Let $\V$ be the family of all pairs $(V,\mu_V)$, where $V\subseteq\V(R,k)$ and
$\mu_V$ is a valuation on $K_V$ extending all the valuations $v$ of $R$, 
$v\in V$. The family $\V$ is nonempty since it contains all the pairs $(\{v\},v)$
for one-element subsets of $\V(R,k)$. It is also ordered by the following order
relation: $(U,\mu_U)\leq (V,\mu_V)$ if and only if $U\subseteq V$ and 
$(\mu_V)|_{K_U}=\mu_U$. Then, the conditions of Zorn's lemma are satisfied, thus
$\V$ contains maximal elements. But by Theorem~\ref{T:tensor} such a maximal
element must coincide with $(\V(R,k),\mu)$ for some valuation $\mu$ on $K$.
Now, let $f_v$ be the natural homomorphism
$$R\to \bigotimes_{v\in\V(R,k)} R= K.$$
Then, $\mu$ induces the valuation $v$ on $R$ via $f_v$.
\end{proof}

\begin{remark}
In the conditions of Theorem~\ref{T:ufield}, the valuation $\mu$ on $K$ naturally
extends to $K/\home(\mu)$, to its field of fractions $Q(K/\home(\mu))$, and,
further, to its algebraic closure $\overline{Q(K/\home(\mu))}$. Thus, the 
universal $k$-algebra $K$ can be assumed to be a domain or even an algebraically
closed field.
\end{remark}

It would be interesting to give an explicit construction of universal valuation 
fields under some reasonable restrictions on the algebra $R$. It may be that such 
a universal valued field depends only on the field $k$ and the valuation $\nu$ 
and works for all $k$-algebras of a certain class. The next proposition gives an
example of such phenomenon. 

\begin{proposition}\label{P:ufforfg}
Let $k$ be an algebraically closed trivially valued field of characteristic $0$, 
and $k(x)=k(x_1,\ldots,x_n)$ the field of rational functions in $n$ variables
with coefficients in $k$. Let $K=\overline{k(x)}(t^\R)$ be the Hahn series field
with coefficients in the algebraic closure of $k(x)$. Then $K$ serves as a 
universal valued field for all finitely generated $k$-algebras $R$ of 
dimension $d\leq n$.
\end{proposition}
\begin{proof}
Let $R$ be a finitely generated $k$-algebra of dimension $d\leq n$. If $v$ is a
valuation of $R$ over $k$, then the residue field $k(v)$ of $v$ has transcendence
degree $r\leq d$ over $k$. Thus $k(v)$ and its algebraic closure can be embedded
to $\overline{k(x)}$. On the other hand, by \cite[Theorem~6]{Kap}, the field
$\overline{k(v)}(t^\R)$ is the maximal valued field with the residue field
$\overline{k(v)}$ and the value group $\R$. It follows that the valuation $v$ is 
induced by a homomorphism from $R$ to $\overline{k(v)}(t^\R)$. It remains to
note that each field of the form $\overline{k(v)}(t^\R)$ embeds to $K$.
\end{proof}

Proposition~\ref{P:ufforfg} suffers obviously from the lack of explicit description
of the algebraic closure $\overline{k(x)}$ of the field of rational functions
in several variables.

\section{Weak universality of the field of Puiseux series}\label{S:wufield}

As in the Introduction, we denote by $K=\bar{k}((t^\Q))=
\bigcup_{N\geq 1}\bar{k}((t^{1/N}))$ the field of Puiseux series with 
coefficients in an algebraically closed field $\bar{k}$, and by $K'=
\bar{k}((t^\R))$ the field of Hahn series, i.e., the set of all 
formal sums
$$a(t)=\sum_{m\in \R} a_m t^m,$$
where $a_m\in\bar{k}$ and $\{m\,|\,a_m\ne 0\}$ is a well-ordered subset of $\R$. 
Let $\mu$ be the natural valuation of $K$ ($K'$).
The ring of Puiseux (respectively Hahn) series $\bar{k}[[t^\Q]]$ 
(resp. $\bar{k}[[t^\R]]$) is the subring of $K$ (resp. $K'$) consisting of the 
series of nonnegative valuation. It is known that if $\bar{k}$ has 
characteristic $0$, then the field $K$ of Puiseux series is algebraically closed 
(\cite{Cohn}); if $\bar{k}=\C$ is the field of complex numbers, this is the 
classical Newton-Puiseux Theorem. The field $K'$ of Hahn series 
is always algebraically closed (\cite[Theorem~1]{MacLane}). 

\begin{theorem}\label{T:wucomplete}
Let $R$ be a complete Noetherian local ring containing a field $k$ isomorphic
to the residue field of $R$. Let $v$ be a local valuation of $R$ and $x_1,\ldots,
x_n$ a finite collection of elements of the maximal ideal of $R$. Then there 
exists a homomorphism $f\colon R\to \mathcal{O}'=\bar{k}[[t^\R]]$ to the ring of 
Hahn series with coefficients in $\bar{k}$ such that for all 
$i=1,\ldots,n$, $v(x_i)=\mu(f(x_i))$. Moreover, if the characteristic of $k$ is 
$0$, the elements $x_1,\ldots,x_n$ analytically generate $R$, and $v(x_i)\in\Q$ 
for all $i=1,\ldots,n$, then there exists a homomorphism $f\colon R\to 
\mathcal{O}=\bar{k}[[t^\Q]]$ to the ring of Puiseux series with coefficients in 
$\bar{k}$ with the same property.
\end{theorem}
\begin{proof}
\emph{Step 1.} First we reduce to the case when $k=\bar{k}$ is algebraically
closed, $R$ is a domain, and do some other preliminary reductions. Consider the 
tensor product $\overline{R}=R\bigotimes_k \bar{k}$. Note that $\overline{R}$ is 
integral over $R$, thus, there exists a prime ideal $\mathfrak{q}\subset
\overline{R}$ such that $\mathfrak{q}\cap R=\mathfrak{p}$, where $\mathfrak{p}=
\home(v)$. The quotient $\overline{R}/\mathfrak{q}$ remains integral over 
$R/\mathfrak{p}$. Hence the local valuation $v$ of $R/\mathfrak{p}$ extends to a 
valuation $\bar{v}$ of $\overline{R}/\mathfrak{q}$, and, as follows from the 
method of Newton polygon (see \cite[Chapitre VI, \S 4, exercise 11]{Bour}), 
$\bar{v}$ is nonnegative on $\overline{R}/\mathfrak{q}$. Passing again to the 
quotient by $\home(\bar{v})$, if necessary, we can assume that $\home(\bar{v})=
\{0\}$. Let $\mathfrak{n}$ be the prime ideal of $\overline{R}/\mathfrak{q}$ 
where $\bar{v}$ is positive. Since $\mathfrak{n}\cap R/\mathfrak{p} =
\mathfrak{m}$ is the maximal ideal of $R/\mathfrak{p}$, the ideal $\mathfrak{n}$ 
is also maximal. Then, consider the localization 
$(\overline{R}/\mathfrak{q})_\mathfrak{n}$ and its completion $R'$ with respect 
to $\mathfrak{n}$. The valuation $\bar{v}$ extends canonically to a local 
valuation $v'$ of $R'$ (see, e.g., \cite[Lemma~5.16]{PPS}). Finally, passing to 
the quotient of $R'$ by $\home(v')$, we can assume that $R'$ is a domain 
and $\home(v')=\{0\}$.

Recall that the \emph{rational rank} of a valuation $v$ on $R$ is the dimension
over $\Q$ of the $\Q$-vector subspace of $\R$ generated by the image of $v$.
Let $r'$ be the rational rank of $v$ and $d$ the Krull dimension of $R$. It 
follows from the Abhyankar inequality (\cite[Th{\'e}or{\`e}me~9.2]{Vaq}) that 
$r'\leq d$. If $r$ is the dimension of the $\Q$-vector subspace of $\R$ 
generated by $v(x_1),\ldots,v(x_n)$, then we have $r\leq r'\leq d$. By the Cohen 
structure theorem \cite[Theorem~7.7]{Eis}, the ring $R$ is a quotient of the ring 
$k[[X_1,\ldots,X_N]]$ of formal power series. Enlarging the finite set 
$\{x_1,\ldots,x_n\}$, if necessary, we can assume that the elements 
$x_1,\ldots,x_n$ analytically generate $R$, i.e., $N=n$ and a surjection from 
$k[[X_1,\ldots,X_n]]$ to $R$ can be chosen so that $X_i\mapsto x_i$, $i=1,\ldots,
n$. Clearly, there is no loss of generality if we also assume that none of
$x_i$ is $0$.

In the reduction steps described above, we performed the following actions
with the ring $R$: we took quotients, considered the tensor product with 
$\bar{k}$, localized at a maximal ideal, and took a completion. These actions 
could only decrease the dimension of $R$. Thus, if the elements $x_1,\ldots,x_n$ 
analytically generate the ring $R$ at the beginning, their images in $R'$ form a 
system of parameters for $R'$. Then, a system of analytical generators for $R'$ 
can be obtained from $x_1,\ldots,x_n$ by adjoining only a finite number of 
elements $x'_1,\ldots,x'_l\in R'$ \emph{integral} over the subring analytically 
generated by $x_1,\ldots,x_n$. It again follows from the method of Newton polygon
that if $v$ takes only rational values on $x_i$, then it also takes only rational
values on $x'_j$, $j=1,\ldots,l$. This shows that the reduction steps are 
applicable also in the ``moreover'' part of the theorem, i.e., we can assume 
that $R$ is a local complete Noetherian domain, $k$ is an algebraically closed
subfield of $R$ of characteristic $0$ and isomorphic to the residue field
of $R$, $x_1,\ldots,x_n$ analytically generate $R$, and $v$ takes only rational
values on $x_i$.

\emph{Step 2.} Now, we reduce to the case $r=d$, where $r$ is the dimension of
the $\Q$-vector subspace of $\R$ generated by $v(x_1),\ldots,v(x_n)$, $d=\dim R$.
Suppose that $r<d$. Then, represent $R$ as a quotient $k[[X_1,\ldots,X_n]]/I$, 
where $I$ is a prime ideal of $k[[X_1,\ldots,X_n]]$. The valuation $v$ induces a 
valuation of $k[[X_1,\ldots,X_n]]$; by abuse of notation, we denote this induced 
valuation also by $v$. Let $w$ be the vector $(w_1,\ldots,w_n)=(v(x_1),\ldots,
v(x_n))\in \R^n$. The proof of the following elementary lemma is left to the 
reader.

\begin{lemma}
Let $w=(w_1,\ldots,w_n)\in \R^n$ be a vector such that the real numbers 
$w_1,\ldots,w_n$ generate a $\Q$-vector subspace of $\R$ of dimension $r$ over 
$\Q$. Then there exists a unique minimal rational vector subspace $L_\Q(w)$ of
$\R^n$ (\emph{rational} means that $L_\Q(w)$ is defined by linear equations with
rational coefficients) such that $w\in L_\Q(w)$. The dimension of $L_\Q(w)$
equals $r$.
\end{lemma}

By general properties of valuations, the initial ideal $\init_w I$ is monomial
free. Consider the local Gr{\"o}bner fan of the ideal $I$ (\cite{BT}).
There exists a rational polyhedral cone $\sigma\subset\R^n$ such that $w$ is 
contained in the relative interior $\mathring{\sigma}$ of $\sigma$ and for all 
$w'\in\mathring{\sigma}$, $\init_{w'} I=\init_w I$. The intersection 
$\sigma\cap L_\Q(w)$ is also a rational polyhedral cone. Let us fix a vector 
$w'\in\mathring{\sigma}\cap L_\Q(w)$ with positive \emph{integral} coordinates.
Note that the initial ideal $\init_{w'}I=\init_w I$ is generated by 
$w$-homogeneous polynomials. Then, let $J=\init_{w'}I\cap k[x_1,\ldots,x_n]$ be 
the corresponding ideal of the polynomial ring. We have
$$k[[x_1,\ldots,x_n]]/\init_{w'}I \simeq (k[x_1,\ldots,x_n]/J)\sphat\,,$$
where the completion is taken with respect to the maximal ideal 
$(x_1,\ldots,x_n)$. Thus, by \cite[Corollary~11.19]{AM}, 
$$\dim k[[x_1,\ldots,x_n]]/\init_{w'}I=\dim k[x_1,\ldots,x_n]/J.$$
But by the standard flat degeneration argument (see, e.g., 
\cite[Lemma~11.10]{PPS}), $\dim k[[x_1,\ldots,x_n]]/\init_{w'}I=\dim R$, and,
since $I$ is supposed to be prime, all minimal associated primes of $\init_{w'}I$
have the same depth $d=\dim R$. We conclude that $\dim k[x_1,\ldots,x_n]/J=d$,
and, moreover, all minimal associated primes of $J$ also have depth $d$.

Since the ideal $J$ is monomial free, it has a point
$x^0\in (k^*)^n$. Choose an irreducible component $Z$ of the zero set of $J$
such that $x^0\in Z$. Let $\gamma_1,\ldots,\gamma_r$ be a basis of the
$\Q$-vector subspace of $\R$ generated by $w_1,\ldots,w_n$. Write
$$w_i=\sum_{j=1}^r w_{ij}\gamma_j,$$
where $i=1,\ldots,n$, $w_{ij}\in\Q$. Rescaling $\gamma_j$ if necessary, we can
even suppose that $w_{ij}$ are integral. By our assumptions, the rank of the
matrix $(w_{ij})_{i=1,j=1}^{n,m}$ is $r$. Since the ideal $J$ is generated by 
$w$-homogeneous polynomials, together with the point $x^0=
(x_{1}^{0},\ldots,x_{n}^{0})$ the variety $Z$ contains also all the points
$$(t_{1}^{w_{11}}\cdots t_{r}^{w_{1r}}x_{1}^{0},\ldots,
t_{1}^{w_{n1}}\cdots t_{r}^{w_{nr}}x_{n}^{0}),\quad (t_1,\ldots,t_r)\in k^r.$$
Thus we can make some $r$ of the coordinates of $x^0$, say, $x_{1}^{0},\ldots,
x_{r}^{0}$, to be equal to $1$. The dimension of the linear affine space $H$:
$x_1=\cdots=x_r=1$ is $n-r$, and the dimension of $Z$ is $d>r$. Hence
$\dim H\cap Z\geq 1$, and there exists at least one point $y^0\ne x^0$,
$y^0\in H\cap Z$. Take a polynomial $\tilde{f}(x_{r+1},\ldots,x_n)$ such that
$\tilde{f}(x_{r+1}^{0},\ldots,x_{n}^{0})=0$, $\tilde{f}(y_{r+1}^{0},\ldots,
x_{n}^{0})\ne 0$. Homogenizing $\tilde{f}$ with respect to the weight vector $w'$
we get a $w'$-homogeneous polynomial $f$ such that $f(x^0)=0$, $f(y^0)\ne 0$.
By construction, $f$ is also $w$-homogeneous, not a monomial, and not contained
in any of the minimal associated primes of $\init_w I$.

\begin{lemma}
Let $I$, $w$, and $f$ be as above. Then, in the ring $k[[x_1,\ldots,x_n]]$, 
$$\init_w(I+(f))=\init_w I + (f).$$
\end{lemma}
\begin{proof}
Consider an arbitrary $fg+h\in (f)+I$, $h\in I$, $g\in k[[x_1,\ldots,x_n]]$. If
$f\cdot\init_w g+\init_w h\ne 0$, then it is the $w$-initial form of $fg+h$ and
is contained in $\init_w I+(f)$. Assume that $f\cdot\init_w g+\init_w h=0$.
Since $f$ is not a zero divisor modulo $\init_w I$, we have $\init_w g\in
\init_w I$. Let $\tilde{g}\in I$ be such that $\init_w\tilde{g}=\init_w g$.
Then $fg+h$ has the same initial form as
$$fg\pm f\tilde{g}+h=fg'+h',$$
where $h'\in I$ and $g'$ has $w$-order strictly greater than $g$. Repeating
this argument we come to the situation when the initial forms of $fg'$ and $h'$
do not cancel, and thus $\init_w(fg+h)\in \init_w I+(f)$. This proves the 
inclusion
$$\init_w(I+(f))\subseteq\init_w I + (f).$$
The inverse inclusion is clear.
\end{proof}

By construction, the ideal $\init_w(I+(f))=\init_w I+(f)$ is monomial free. By
Bergman's Theorem~\ref{T:Bergman}, the monomial valuation $v'$ on 
$k[[x_1,\ldots,x_n]]$ defined by $v'(x_i)=w_i$, $i=1,\ldots,n$, can be 
transformed to a valuation $\bar{v}$ such that $\bar{v}(x_i)=w_i$, $i=1,\ldots,n$,
and $\bar{v}|_{I+(f)}=+\infty$. Therefore we can pass to the ring $R'=R/(f)$,
$\dim R'=d-1$.

\emph{Step 3.} It remains to consider the case $d=\dim R=r$. The rational rank
$r'$ of the valuation $v$ is not greater than $d$ by Abhyankar inequality, so
in this case $r=r'$. Choose a system of parameters $y_1,\ldots,y_d$ for $R$
such that $v_1=v(y_1)$, $\ldots$, $v_d=v(y_d)$ are linearly independent over
$\Q$. The ring $R$ is integral over $k[[y_1,\ldots,y_d]]$.

Embed $k[[y_1,\ldots,y_d]]$ to the ring of Hahn series $\mathcal{O}'
\subset k((t^\R))$ by sending $y_i$ to $t^{v_i}$, $i=1,\ldots,d$. The natural 
valuation $\mu$ of $k((t^\R))$ is an extension of $v$ from $k[[y_1,\ldots,y_d]]$ 
via this embedding. By \cite[Chapter XII, \S 3]{Lang}, every extension of $v$ 
from $k[[y_1,\ldots,y_d]]$ to $R$ can be induced by an embedding of $R$ to 
$$\overline{Q(k[[y_1,\ldots,y_d]])_v},$$
i.e., to the algebraic closure of the completion with respect to $v$ of the field
of fractions of $k[[y_1,\ldots,y_d]]$. Since $k((t^\R))$ is algebraically closed 
and complete, we conclude that there exists an embedding of $R$ to $k((t^\R))$ 
such that $\mu$ induces $v$. Moreover, the method of Newton polygon shows that 
the image of $R$ is contained in $\mathcal{O}'$. 

Finally, assume that $v$ takes only rational values on $x_i$ and the field $k$
has characteristic $0$. This implies, in particular, that $r=d=1$. This time
we embed $k[[y]]=k[[y_1]]$ to the ring of Puiseux series $\mathcal{O}$ by 
sending $y$ to $t^{v_1}$. As above, we get an embedding of $R$ to $k((t^\R))$
inducing the valuation $v$, but, since the field $k((t^\Q))$ is also 
algebraically closed, the image is contained in $k((t^\Q))$ and, by the method
of Newton polygon, in $\mathcal{O}$.
\end{proof}

The following theorem shows weak universality of the field of Hahn series with 
respect to valuations on finitely generated algebras.

\begin{theorem}
Let $R$ be a finitely generated algebra over a field $k$. Let $v$ be a 
valuation of $R$ and $x_1,\ldots,x_n$ a finite collection of elements of $R$. 
Then there exists a homomorphism $f\colon R\to \bar{k}((t^\R))$ to the field of 
Hahn series with coefficients in $\bar{k}$ such that for all 
$i=1,\ldots,n$, $v(x_i)=\mu(f(x_i))$. Moreover, if the characteristic of $k$ is 
$0$, the elements $x_1,\ldots,x_n$ generate $R$ over $k$, and $v(x_i)\in\Q$ 
for all $i=1,\ldots,n$, then there exists a homomorphism $f\colon R\to
\bar{k}((t^\Q))$ to the field of Puiseux series with coefficients in $\bar{k}$ 
with the same property.
\end{theorem}
\begin{proof}
This theorem follows from the results on lifting points in tropical varieties,
see, e.g., \cite{JMM}. Alternatively, it can be proven by an argument parallel
to the proof of Theorem~\ref{T:wucomplete}.
\end{proof}

\section{Lifting points in local tropical varieties}\label{S:liftpoint}

In this section we recall briefly the definition of local tropicalization
following \cite{PPS} and deduce the lifting points property of local tropical 
varieties from Theorem~\ref{T:wucomplete}. Let $M$ be a finitely generated
free abelian group and $M_\R$ the vector space $M\bigotimes_Z \R$. For 
rational polyhedral cones $\check\sigma$ in $M_\R$, we consider additive
semigroups $\Gamma$ of the form $\check\sigma\cap M$ or, more generally,
finitely generated subsemigroups (with neutral element) of $\check\sigma\cap M$ 
such that $\sat(\Gamma)$, the saturation of $\Gamma$ (see 
\cite[Definition~2.10]{PPS}), equals $\check\sigma\cap M$. Each semigroup 
homomorphism from $\Gamma$ to $\R$ lifts uniquely to a group homomorphism from $M$
to $\R$. Thus we can identify the group $\Hom_{sg}(\Gamma,\R)$ of all semigroup 
homomorphisms from $\Gamma$ to $\R$ with the dual space of $M_\R$; we denote this 
dual space by $N_\R$, and by $N$ the dual lattice of $M$. The semigroup 
homomorphisms from $\Gamma$ to the extended real line $\Ri=\R\cup\{+\infty\}$ 
live in a certain partial compactification of $N_\R$ that is denoted by 
$L(\sigma,N)$ in \cite[Section~4]{PPS} and called the \emph{linear variety} 
corresponding to $\sigma$ and $N$; here $\sigma$ is the dual cone of 
$\check\sigma$. The cone $\sigma$ corresponds to nonnegative homomorphisms from 
$\Gamma$ to $\R$; the nonnegative homomorphisms from $\Gamma$ to $\Ri$ form a 
certain subspace of $L(\sigma,N)$ denoted $\overline{\sigma}$. We denote by 
$\overline{\sigma}^\circ$ the interior of $\overline{\sigma}^\circ$.

Now, let $R$ be a local ring with the maximal ideal $\mathfrak{m}$. Let 
$\V_{\geq 0}(R)$ be the set of all nonnegative real valuations of $R$ and 
$\V_{>0}(R)$ the set of all local valuations of $R$, i.e., valuations that are 
nonnegative on $R$ and positive on the maximal ideal $\mathfrak{m}$. Let 
$\gamma\colon\Gamma\to R$ be a homomorphism from $\Gamma$ to the multiplicative 
semigroup of $R$ and assume that none of the elements of 
$\gamma^{-1}(\mathfrak{m})$ is invertible in $\Gamma$. Then, for each 
$v\in\V_{\geq 0}$ (resp. $v\in\V_{>0}$), the composition $v\circ\gamma$ is an 
element of $\overline{\sigma}$ (resp. $\overline{\sigma}^\circ$). In this way we 
get a map
$$\trop\colon \V_{\geq 0}(R)\to \overline{\sigma}\quad
(\text{resp. } \trop\colon \V_{>0}(R)\to \overline{\sigma}^\circ),$$
which we call the \emph{tropicalization map}.

\begin{definition}
The \emph{local nonnegative tropicalization} of the morphism $\gamma$, denoted
$\nntrop(\gamma)$, is the image of $\V_{\geq 0}(R)$ in $\overline{\sigma}$ under 
the tropicalization map $\trop$. The \emph{local positive tropicalization}, 
denoted $\ptrop(\gamma)$, is the closure in $\overline{\sigma}^\circ$ of the 
image of $\V_{>0}(R)$ under the tropicalization map.
\end{definition}

Next we restrict to local rings $R$ of special type. Suppose that the cone
$\check\sigma$ is \emph{strictly convex}, in particular, the only invertible
element of $\Gamma$ is $0$. Let $k$ be a field and consider all possible
formal series
$$\sum_{m\in\Gamma} a_m \chi^m,$$
where $a_m\in k$. Any two such series can be added and multiplied in a natural 
way, in particular, $\chi^m\cdot\chi^{m'}=\chi^{m+m'}$. These series form a 
complete Noetherian local ring $k[[\Gamma]]$ called the \emph{ring of formal 
power series} over $\Gamma$, see \cite[Section~8]{PPS}. There is a natural 
semigroup homomorphism $\Gamma\to k[[\Gamma]]$ defined by $\Gamma\ni 
m\mapsto \chi^m$, the elements $\chi^m\in k[[\Gamma]]$ being called the 
\emph{monomials}. If $I\subset k[[\Gamma]]$ is an ideal, let $\gamma$ be the 
composition of the natural maps
$$\Gamma\to k[[\Gamma]]\to R=k[[\Gamma]]/I.$$
Now, we can consider the nonnegative (resp. positive) tropicalization 
$\nntrop(\gamma)$ (resp. $\ptrop(\gamma)$) of $\gamma$. In the situation just
described, we denote also $\nntrop(\gamma)=\nntrop(I)$ ($\ptrop(\gamma)=\ptrop(I)$)
and call it the \emph{local nonnegative} (resp. \emph{positive}) 
\emph{tropicalization of the ideal} $I$.

It is proven in \cite[Theorem~11.9]{PPS} that both nonnegative $\nntrop(I)
\subseteq\overline{\sigma}$ and positive $\ptrop(I)\subseteq
\overline{\sigma}^\circ$ tropicalizations are rational PL conical subspaces of
real dimension $d$ equal to the Krull dimension of $R$. This means, in particular,
that the part of $\nntrop(I)$ contained in $N_\R$ is a support of a rational
(with respect to the lattice $N$) polyhedral fan. The nonnegative tropicalization
$\nntrop(I)$ is stratified by the positive tropicalizations $\ptrop(I_\tau)$
of certain truncations of the ring $R$, see \cite[Section~12, in particular
Lemma~12.9]{PPS} for the details. 

\begin{theorem}[Lifting points lemma for local tropical varieties]
\label{T:liftpoint}
Let $I\subset k[[\Gamma]]$ be an ideal. If $u\in\ptrop(I)$, then there exists a 
homomorphism $f\colon R=k[[\Gamma]]/I\to \bar{k}[[t^\R]]$ to the ring of
Hahn series with coefficients in $\bar{k}$ such that 
$u=\trop(\mu\circ f)$, where $\mu$ is the natural valuation of $\bar{k}[[t^\R]]$.
Moreover, if the field $k$ has characteristic $0$ and the point $u$ is rational, 
then there exists a homomorphism $f\colon R\to \bar{k}[[t^\Q]]$ to the ring of 
Puiseux series with the same property. 
\end{theorem}
\begin{proof}
By \cite[Theorem~11.2]{PPS}, the image of $\V_{>0}(R)$ in 
$\overline{\sigma}^\circ$ is closed under the tropicalization map, hence there 
exists $v\in\V_{>0}(R)$ such that $u=\trop(v)$. If $u\in \ptrop(I)\subseteq
\overline{\sigma}^\circ$ belongs to some stratum of $\overline{\sigma}$ at 
infinity, then $v$ is induced by a valuation on a quotient $R/J$ of the ring
$R$, where $R/J$ also has the form $k[[\Gamma']]$ for some semigroup $\Gamma'$, 
see \cite[the end of Section~8]{PPS}. Therefore it suffices to consider the case 
$u\in N_\R$. Choose monomials $x_1,\ldots,x_n\in\Gamma$ which generate the 
semigroup $\Gamma$. A semigroup homomorphism from $\Gamma$ to $\R$ is uniquely 
determined by the images of $x_1,\ldots,x_n$. By Theorem~\ref{T:wucomplete} there
exists a homomorphism $f\colon R\to \bar{k}[[t^\R]]$ such that $\mu(f(x_i))=
u(x_i)$ for all $i=1,\ldots,n$. If the characteristic of $k$ is $0$ and $u$ is 
rational, a homomorphism $f$ can be chosen so that the target ring is 
$\bar{k}[[t^\Q]]$. This $f$ is the required homomorphism.
\end{proof}

\begin{corollary}
Let $I\subset k[[\Gamma]]$ be an ideal. The following three definitions of the
local positive tropicalization of $I$ are equivalent:
\begin{itemize}
\item[1)] $\ptrop(I)$ is the image of $\V_{>0}(k[[\Gamma]]/I)$ under the
tropicalization map;
\item[2)] $\ptrop(I)$ is the set of all $w\in\overline{\sigma}^\circ$ such that 
the initial ideal $\init_w I$ is monomial free;
\item[3)] $\ptrop(I)$ is the set $\{\trop(\mu\circ f)\,|\,f\colon R\to 
\bar{k}[[t^\R]]\}$, where $f$ runs over all local homomorphisms from
$R=k[[\Gamma]]/I$ to $\bar{k}[[t^\R]]$.
\end{itemize}
If the field $k$ has characteristic $0$, then $\ptrop(I)$ can also be described
as
\begin{itemize}
\item[$3'$)] the closure of the set $\{\trop(\mu\circ f)\,|\,f\colon R\to 
\bar{k}[[t^\Q]]\}$ in $\overline{\sigma}^\circ$, where $f$ runs over all local 
homomorphisms from $R$ to $\bar{k}[[t^\Q]]$.
\end{itemize}
\end{corollary}
\begin{proof}
Equivalence of 1) and 2) is proven in \cite[Theorem~11.2]{PPS}. Equivalence of
1) and 3), and, if the characteristic of $k$ is $0$, of 1) and $3'$), follows 
from Theorem~\ref{T:liftpoint} and from the fact that the local positive 
tropicalization defined by 1) is a \emph{rational} conical PL space, 
\cite[Theorem~11.9]{PPS}.
\end{proof}

\end{document}